\documentclass{article}
\usepackage[english]{babel}
\usepackage[utf8]{inputenc}
\usepackage{amsmath}
\usepackage{graphicx}
\usepackage{amsfonts}
\usepackage{enumitem}
\usepackage{amssymb}
\usepackage[colorinlistoftodos]{todonotes}
\usepackage{geometry}
\usepackage{amsthm}

\newcommand{\floor}[1]{\left\lfloor #1 \right\rfloor}

\newcommand{\set}[1]{\left\{ #1 \right\}}

\newcommand{\ol}{\overline}

\newcommand{\ZZ}{\mathbb{Z}}

\newtheorem{thm}{Theorem}[section]
\newtheorem{lem}[thm]{Lemma}

\newtheorem{conj}[thm]{Conjecture}

\theoremstyle{definition}
\newtheorem{defn}[thm]{Definition}

\DeclareMathOperator{\width}{width}

\title{Everywhere unbalanced configurations}
\author{David Conlon\thanks{Department of Mathematics, Caltech, Pasadena, CA 91125, USA. Email: {\tt dconlon@caltech.edu}. Research supported by NSF Awards DMS-2054452 and DMS-2348859.} \and Jeck Lim\thanks{Department of Mathematics, Caltech, Pasadena, CA 91125, USA. Email: {\tt jlim@caltech.edu}. Research partially supported by an NUS Overseas Graduate Scholarship.}}
\date{}

\begin{document}
\maketitle

\begin{abstract}
An old problem in discrete geometry, originating with Kupitz, asks whether there is a fixed natural number $k$ such that every finite set of points in the plane has a line through at least two of its points where the number of points on either side of this line differ by at most $k$. We give a negative answer to a natural variant of this problem, showing that for every natural number $k$ there exists a finite set of points in the plane together with a pseudoline arrangement such that each pseudoline contains at least two  points and there is a pseudoline through any pair of points where the number of points on either side of each pseudoline differ by at least $k$. 
Moreover, we may find such a configuration with 
at most $2^{2^{ck}}$ points, which, by a result of Pinchasi, is best possible up to the value of the constant $c$.
\end{abstract}

\section{Introduction}

Does there exist a positive integer $k$ such that for every finite set of points in the plane there is a line containing at least two of these points where the number of points on either side of this line differ by at most $k$? This basic problem was first raised by Kupitz~\cite{K79} in the late 1970s and has since been reiterated by many authors, including Alon~\cite{Al02}, Erd\H{o}s~\cite{Erd84}, Green~\cite{Green}, Kalai~\cite{Kalai}, Pach~\cite{P18} and Pinchasi~\cite{P03}. It was also singled out for inclusion in the book of Brass, Moser and Pach~\cite{BMP05} that surveys many of the most interesting open problems in discrete geometry.

If a point set has an odd number of points with no three of them on a line, then the number of points on either side of each line through two of the points must differ by at least one. Kupitz's original conjecture~\cite{K79} was that this should be the extremal case, that is, that there should always be a line containing at least two points of any finite point set where the number of points on either side differ by at most one. This was disproved by Alon~\cite{Al02}, who showed that there are finite point sets where the number of points on either side of each line determined by the set differ by at least two. No better example is known, though in Figure~\ref{fig:odd} we give a different example to Alon's with the same property, but, unlike his examples, containing an odd number of points. This example can easily be extended to any larger odd number of points by adding an equal number of points on both ends of some line.

\begin{figure}
    \centering
    \includegraphics[scale=0.7]{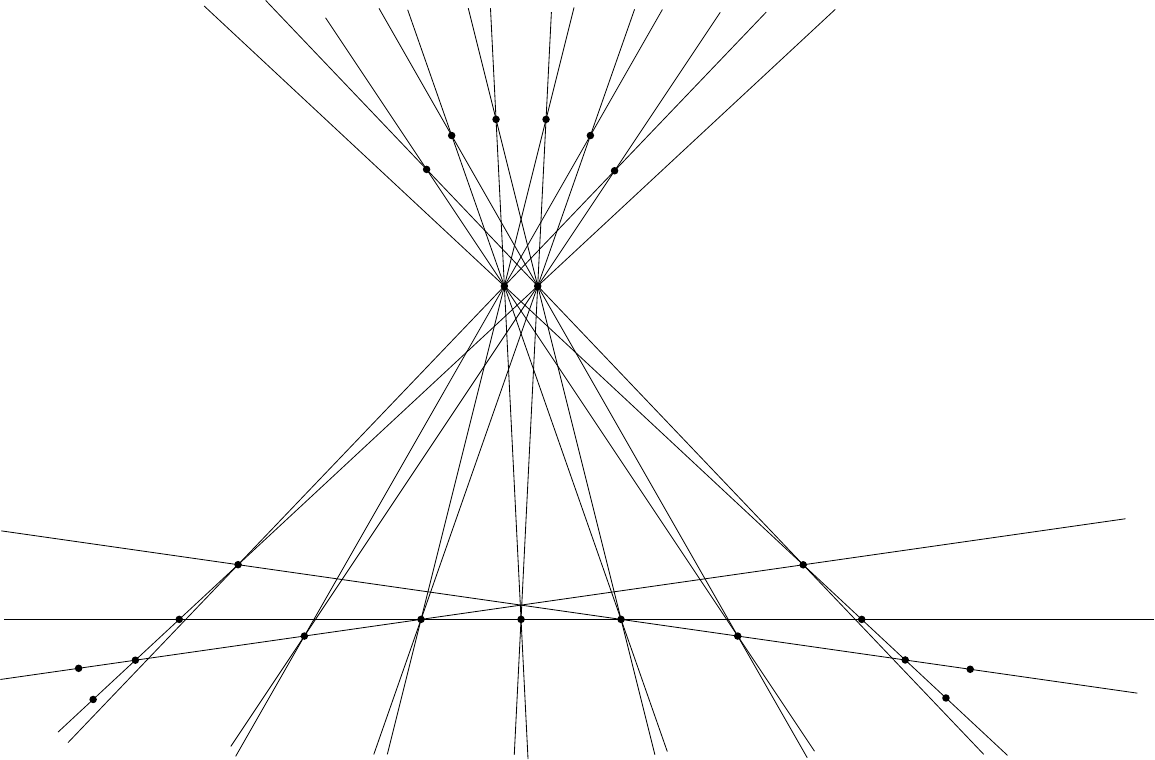}
    \caption{A set of 23 points where the number of points on either side of each line determined by the set differ by at least two.}
    \label{fig:odd}
\end{figure}

In the opposite direction, improving earlier unpublished results of Alon and Perles, it was shown by Pinchasi~\cite{P03} that there exists an absolute constant $C$ such that every $n$-point set determines a line where the number of points on either side of the line differ by at most $C \log \log n$. In fact, his result applies in the much broader context of generalised configurations. Recall that a {\it pseudoline arrangement} is a collection of two-way unbounded simple curves any two of which meet in at most one point. A {\it generalised configuration} is then a finite set of points in the plane together with a pseudoline arrangement such that each pseudoline contains at least two points and there is a pseudoline through any pair of points. With this terminology, we may now state Pinchasi's result, of which the result stated above is clearly a special case.

\begin{thm}[Pinchasi] \label{thm:Pin}
There exists an absolute constant $C$ such that every generalised configuration with $n$ points contains a pseudoline where the number of points on either side of the pseudoline differ by at most $C \log \log n$.
\end{thm}

Our main result says that Pinchasi's result is tight up to the constant.

\begin{thm} \label{thm:main1}
There exists a positive constant $c$ such that for every sufficiently large natural number $n$ there is a generalised configuration with $n$ points where the number of points on either side of each pseudoline differ by at least $c \log \log n$. 
\end{thm}

In particular, for any positive integer $k$, there is a generalised configuration where the number of points on either side of each pseudoline differ by at least $k$. 

The reason why Pinchasi's result applies to pseudolines as well as lines is that he works throughout with {\it allowable sequences of permutations}. This is a sequence of permutations of $[n] := \{1, 2, \dots, n\}$ starting with the identity permutation $1, 2, \dots, n$ and ending with its reverse $n, n-1, \dots, 1$ where every permutation in the sequence arises from its predecessor by flipping the elements in one or more non-overlapping increasing subsequences of consecutive elements. For example, the sequence
\[12345 \rightarrow 21354 \rightarrow 25314 \rightarrow 52341 \rightarrow 54321\]
is an allowable sequence of permutations of $[5]$.

Introduced by Goodman and Pollack~\cite{GP80}, the importance of allowable sequences is that every generalised configuration of points gives rise to such a sequence and, conversely~\cite{GP84}, that every allowable sequence can be realised as a generalised configuration. Some intuition can be gained by thinking about the generalised configuration consisting of a point set and the line arrangement determined by that point set. If we fix another line $\ell$ such that the orthogonal projection of the points of our set onto $\ell$ are all distinct, we can label the points as $1, 2, \dots, n$ in that order. If we now rotate $\ell$ counterclockwise about a fixed point, then the order of the orthogonal projections of our points onto $\ell$ shifts as we rotate, jumping to a different permutation every time $\ell$ moves through the perpendicular to any line in our collection, ultimately arriving at $n, n-1, \dots, 1$ after rotating through $180^{\circ}$. The construction of an allowable sequence when we instead have a pseudoline arrangement is similar, though we refer the reader to~\cite{FG18} for more details.

In this language, Pinchasi's result may be stated as saying that there is an absolute constant $C$ such that every allowable sequence of permutations of $[n]$ uses a flip $[a,b]$ where $|\frac{a+b}{2} - \frac{n+1}{2}| \leq C \log \log n$, where here the {\it flip} $[a,b]$ is understood as taking a permutation $\pi_1, \pi_2, \dots, \pi_n$ of $[n]$ and reversing the elements in the block $\pi_a, \pi_{a+1}, \dots, \pi_b$. Similarly, the result that we will actually prove can be stated as follows. 

\begin{thm} \label{thm:main2}
There exists a positive constant $c$ such that for every sufficiently large natural number $n$ there is an allowable sequence of permutations of $[n]$ where each flip $[a,b]$ has $|\frac{a+b}{2} - \frac{n+1}{2}| \geq c \log \log n$.
\end{thm}

As already mentioned, a result of Goodman and Pollack~\cite{GP84} says that any allowable sequence of permutations of $[n]$ can be realised as a generalised configuration with $n$ points. Moreover, when we flip $[a,b]$, the number of elements on either side of the corresponding pseudoline must differ by at least $|n - b - a + 1|$, so Theorem~\ref{thm:main1} follows as an immediate corollary of Theorem~\ref{thm:main2}. As such, the remainder of the paper will be concerned with proving this latter theorem. We will prove the theorem for sufficiently large odd $n$, but it trivially extends to all large $n$ by removing a single point when necessary.

\section{Centred sequences, flips and blocks}

We begin our proof of Theorem~\ref{thm:main2} by introducing some notation and justifying some standard procedures that will be used repeatedly in the construction. We first fix a positive integer $t$ that will appear throughout the argument.

\begin{defn}
A \textit{centred sequence} $\mathbf{A}$ is an injective map $\mathbf{A}:[a,b]\to\ZZ$ for some integers $a\leq b$. For $a\leq n\leq b$, we denote $\mathbf{A}(n)$ by $\mathbf{A}_n$. The \textit{underlying sequence} of $\mathbf{A}$ is then $\mathbf{A}_{a},\mathbf{A}_{a+1},\ldots,\mathbf{A}_b$.

Given such a centred sequence $\mathbf{A}$, a \textit{flip} $F$ is an interval $[c,d]$, where $a\leq c\leq d\leq b$. The \textit{size} of $F$ is $d-c+1$.
Performing the flip $F$ on the sequence $\mathbf{A}$ gives a new centred sequence $\mathbf{A}':[a, b]\to \ZZ$ where the subsequence $[c,d]$ of $\mathbf{A}$ is reversed.

The flip $F$ is \textit{valid} if the centred subsequence $\mathbf{A}|_{[c,d]}$ is increasing and $\frac{c+d}{2}$ does not lie in the real interval $[-t,t]$. 
\end{defn}

Our aim will be to find some $n>t$ and, starting from the identity centred sequence $\mathbf{A}:[-n,n]\to\ZZ$, to perform a sequence of valid flips, ending at the reverse of $\mathbf{A}$.

\begin{defn} \label{def:block}
A \textit{block} $B$ is a sequence $B_1,\ldots,B_n$ of distinct integers. The \textit{size} of $B$, which we denote by $|B|$, is $n$.

We can again define a \textit{flip} $[c,d]$ of $B$, where performing this flip will reverse the subsequence $B_c,B_{c+1},\ldots,B_d$. We say that a flip is \textit{valid} if the subsequence $B_c,\ldots,B_d$ is initially increasing.

Denote by $\ol{B}$ the block which is the reverse of $B$. If $|B|=m$, write $\mathbf{B}$ for the centred sequence $\mathbf{B}:[t-m+1,t]\to\ZZ$ with the same underlying sequence as $B$. Conversely, if $\mathbf{A}$ is a centred sequence, we write $A$ for the block corresponding to the underlying sequence of $\mathbf{A}$.
\end{defn}

\begin{defn}
We say that a centred sequence $\mathbf{A}'$ can be \textit{obtained} from a centred sequence $\mathbf{A}$ if there is a sequence of valid flips sending $\mathbf{A}$ to $\mathbf{A}'$. We say that a block $B'$ can be \textit{obtained} from a block $B$ if there is a sequence of valid flips sending $B$ to $B'$. We also say that we can \textit{go from $\mathbf{A}$ to $\mathbf{A}'$} or that we can \textit{go from $B$ to $B'$}.
\end{defn}

Note that a valid flip of a block is equivalent to making several valid flips of size 2 each. Thus, to see if $B'$ can be obtained from $B$, it suffices to consider only valid flips of size 2.

\begin{defn}
Given two blocks $B,C$ of sizes $m,n$, their \textit{concatenation} $BC$ is the block of size $m+n$ with sequence $B_1,\ldots,B_m,C_1,\ldots,C_n$.

Given a centred sequence $\mathbf{A}:[a,b]\to\ZZ$, the \textit{concatenation} $\mathbf{A}B$ is the centred sequence $\mathbf{A}':[a,b+m]\to\ZZ$ with underlying sequence $\mathbf{A}_{a},\ldots,\mathbf{A}_b,B_1,\ldots,B_m$. We can similarly concatenate a block to the left, writing it as $B\mathbf{A}$.
\end{defn}

\begin{defn}
Given blocks or centred sequences $B,C$, we write $B\prec C$ if $\max_i B_i < \min_j C_j$. We also write $B\succ 0$ if $B_i>0$ for all $i$ and similarly for $B\prec 0$. 

We say that a block or centred sequence $B$ is \textit{increasing} (resp., \textit{decreasing}) if its underlying sequence is increasing (resp., decreasing).
\end{defn}

Observe that if $B,C$ are blocks such that $B\prec C$, then we can easily go from $BC$ to $CB$. We will constantly use this simple operation in what follows. We now describe another basic operation that we use repeatedly.

\begin{lem}[Shifting] \label{lem:shift}
For all $n\geq 3^{2t}$, the following holds. Let $\mathbf{A}:[-t,t]\to \ZZ$ be a centred sequence and $B,C$ be blocks such that
\begin{itemize}
    \item $|B|=n$, $|C|=2t+1$,
    \item $B$ is increasing,
    \item $\mathbf{A}\prec B\prec C$.
\end{itemize}
Then one can go from $\mathbf{A}BC$ to $\mathbf{C}D$, where $D$ is decreasing and, in keeping with Definition~\ref{def:block}, $\mathbf{C}:[-t,t]\to\ZZ$ is the centred sequence with the same underlying sequence as the block $C$. 
\end{lem}

\begin{figure}
    \centering
    \includegraphics[scale=1.0]{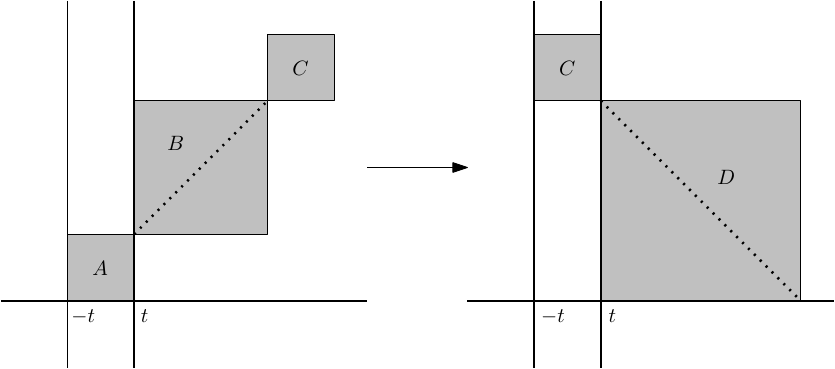}
    \caption{A graphical depiction of Lemma~\ref{lem:shift}, Shifting.}
    \label{fig:shift}
\end{figure}

The idea here is that shifting a block $C$ into the forbidden region is difficult and requires ``ammunition'' in the form of an increasing block $B$. Moreover, after we ``use up'' the ammunition $B$, the resulting decreasing block $D$ is less useful to us.

A graphical depiction of this lemma can be seen in Figure~\ref{fig:shift}. Throughout the paper, we will make extensive use of such figures to illustrate the steps in our construction. These figures will also accurately depict which blocks lie above which other blocks.

\begin{proof}
We will show by induction on $k=t,t-1,\ldots,-t$ that there is some $N_k$ such that, for all $n\geq N_k$, the following holds. Let $\mathbf{A}:[k,t]\to\ZZ$ be a centred sequence and $B,C$ be blocks such that
\begin{itemize}
    \item $|B|=n$, $|C|=t-k+1$,
    \item $B$ is increasing,
    \item $\mathbf{A}\prec B\prec C$.
\end{itemize}
Then one can go from $\mathbf{A}BC$ to $\mathbf{C}D$, where $D$ is decreasing and $\mathbf{C}:[k,t]\to\ZZ$ is the centred sequence with the same underlying sequence as the block $C$.

The base case $k=t$ is simple: we may take $N_t=0$ and perform a single flip on the whole centred sequence $\mathbf{A}BC$. 

For $-t\leq k<t$, set $N_k=2(N_{k+1}+t-k)$. Write $\mathbf{A}=A_1\mathbf{A}'$, where $A_1$ is viewed as a block of size 1 and $\mathbf{A}'$ is the remaining centred sequence. Write $C=C_1C'$, where $C_1$ is a block of size 1 and $C'$ is the remaining block, and  $B=B^1B^2B^3B^4$, where $|B^1|=N_{k+1}$ and $|B^2|=|B^3|=t-k$, so that $|B^4|\geq N_{k+1}$. Apply the induction hypothesis for $k+1$ on $\mathbf{A}'B^1B^2$ to obtain $\mathbf{B^2}E^1$, where $E^1$ is decreasing (see Figure~\ref{fig:shift_pf1}). 

\begin{figure}
    \centering
    \includegraphics[scale=0.9]{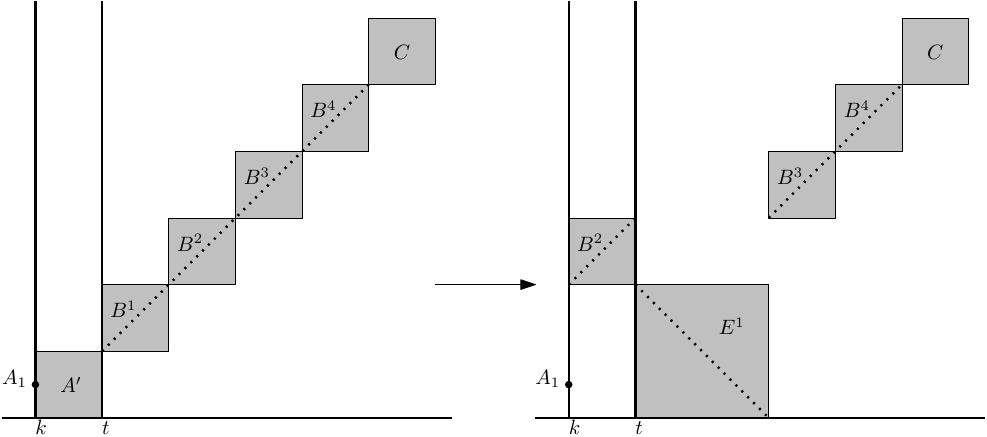}
    \caption{Proof of Lemma~\ref{lem:shift}, part 1.}
    \label{fig:shift_pf1}
\end{figure}

Since $E^1\prec B^3$, we may swap them and go from $E^1 B^3$ to $B^3 E^1$. By continuing to freely swap blocks this way, we can go from $A_1 \mathbf{B^2}  E^1 B^3 B^4 C$ to $A_1 \mathbf{B^2}  B^3 C_1 E^1 B^4 C'$. 
We then perform the flip $[k,2t-k+1]$ to get 
$$C_1 \mathbf{\ol{B^3}}\ \ol{B^2} A_1 E^1 B^4 C'$$ (see Figure~\ref{fig:shift_pf2}).

\begin{figure}
    \centering
    \includegraphics[scale=1.0]{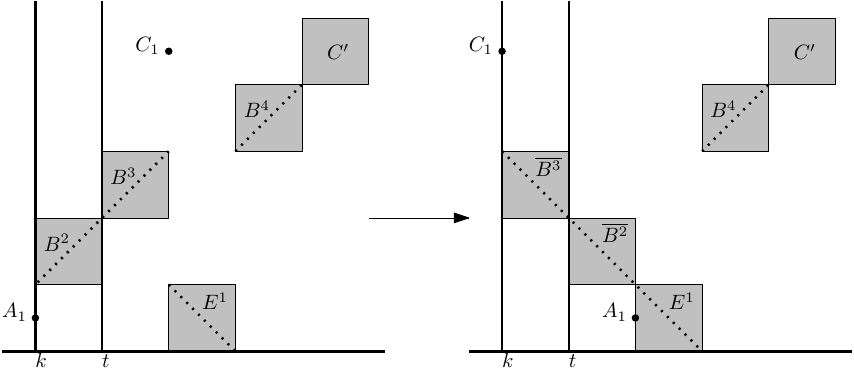}
    \caption{Proof of Lemma~\ref{lem:shift}, part 2.}
    \label{fig:shift_pf2}
\end{figure}

We can go from $\ol{B^2} A_1 E^1$ to $E^2$, where $E^2$ is decreasing. Then we go from $E^2 B^4 C'$ to $B^4 C' E^2$ to obtain
$$C_1 \mathbf{\ol{B^3}} B^4 C' E^2.$$ 
Finally, we again apply the induction hypothesis for $k+1$ to $\mathbf{\ol{B^3}} B^4 C'$ to obtain
$$\mathbf{C} E^3 E^2,$$
so we can set $D=E^3 E^2$, which is decreasing (see Figure~\ref{fig:shift_pf3}).

\begin{figure}
    \centering
    \includegraphics[scale=1.0]{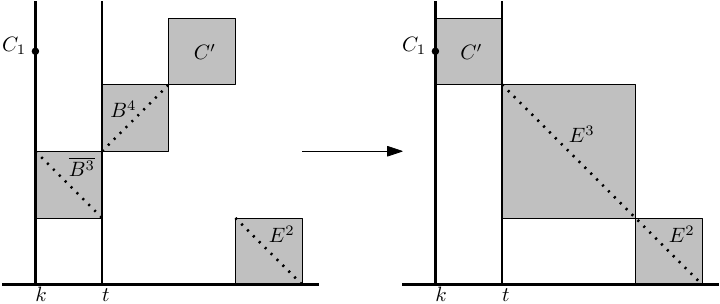}
    \caption{Proof of Lemma~\ref{lem:shift}, part 3.}
    \label{fig:shift_pf3}
\end{figure}

To finish the proof, note that a simple induction shows that $N_k\leq 3^{t-k}$. In particular, $N_{-t}\leq 3^{2t}$, as required.
\end{proof}

When we apply Lemma~\ref{lem:shift}, we will simply say that we \textit{shift} $C$. With one additional step, we also obtain a procedure for bringing a block across 0.

\begin{lem}[Reflection] \label{lem:reflect}
For all $n\geq 3^{2t}+4t+2$, the following holds. Let $\mathbf{A}:[-t,t]\to \ZZ$ be a centred sequence and $B$, $C$, $X$ be blocks such that
\begin{itemize}
    \item $|B|=n$,
    \item $|C|=|X|$,
    \item $B$, $C$, $X$ are increasing,
    \item $X\prec \mathbf{A}\prec B\prec C$.
\end{itemize}
Then one can go from $X \mathbf{A} B C$ to $\ol{C} \mathbf{D} E$, where $\mathbf{D}:[-t,t]\to \ZZ$ is a decreasing centred sequence (with the same underlying sequence as the last $2t+1$ elements of $B$ in reverse), 
$E$ is decreasing and $\mathbf{D}\succ E$ (see Figure~\ref{fig:reflect}). 
\end{lem}

\begin{figure}
    \centering
    \includegraphics[scale=1.0]{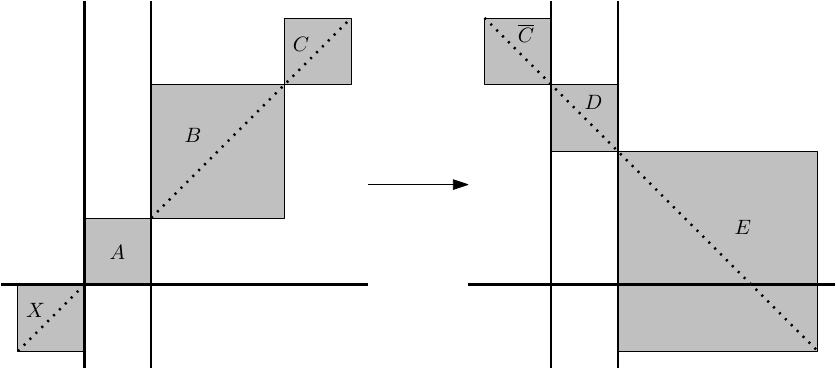}
    \caption{Reflection.}
    \label{fig:reflect}
\end{figure}

\begin{proof}
Decompose $B=B^1 B^2 B^3$, where $|B^2|=|B^3|=2t+1$. Then, $|B^1|\geq 3^{2t}$, so, by shifting $B^2$ using Lemma~\ref{lem:shift}, we can go from $\mathbf{A} B C$ to
$$\mathbf{B^2} D' B^3 C,$$
where $\mathbf{B^2}:[-t,t]\to\ZZ$ and $D'$ is decreasing (see Figure~\ref{fig:reflect1}). 

\begin{figure}
    \centering
    \includegraphics[scale=1.0]{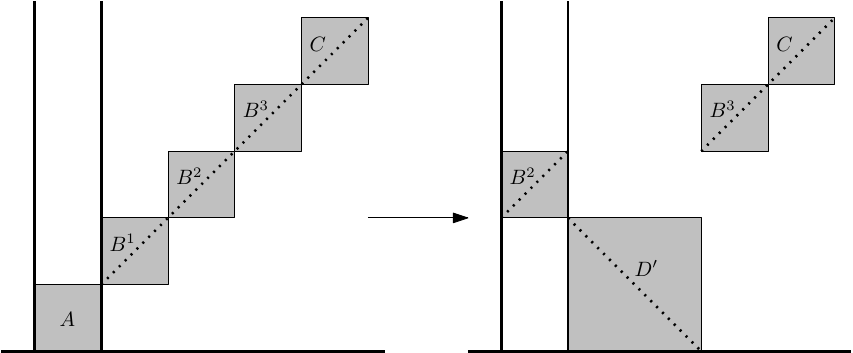}
    \caption{Proof of Lemma~\ref{lem:reflect}, part 1.}
    \label{fig:reflect1}
\end{figure}

\begin{figure}
    \centering
    \includegraphics[scale=0.9]{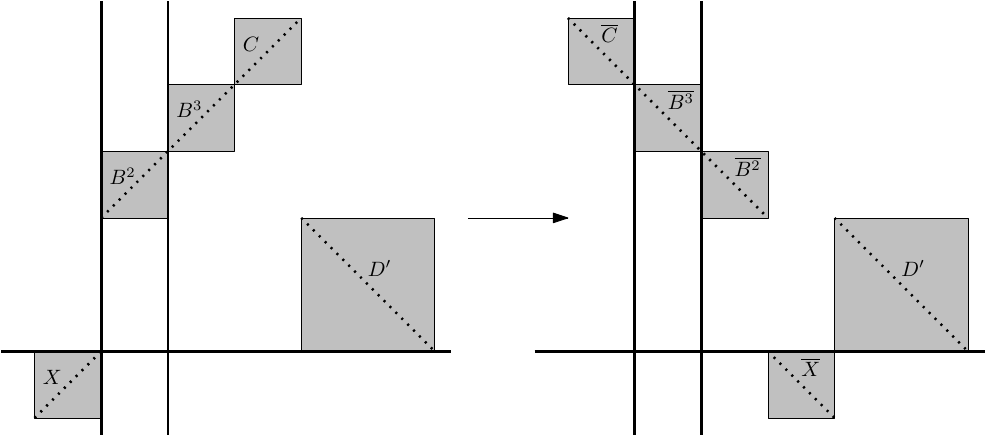}
    \caption{Proof of Lemma~\ref{lem:reflect}, part 2.}
    \label{fig:reflect2}
\end{figure}

Thus, after moving $D'$ to the right, we can go from $X \mathbf{A} B C$ to $X \mathbf{B^2} B^3 C D'$. Now perform the flip $[-t-|C|,3t+1+|C|]$ to get
$$\ol{C}\ \mathbf{\ol{B^3}}\ \ol{B^2}\ \ol{X}  D',$$
from which we can go to $\ol{C} \mathbf{D} E$, where $\mathbf{D}=\mathbf{\ol{B^3}}$ and $E=\ol{B^2} D' \ol{X}$ (see Figure~\ref{fig:reflect2}).
\end{proof}

When we apply Lemma~\ref{lem:reflect}, we will simply say that we \textit{reflect} $C$. 

To close this section, we note that the operations above also hold if we mirror the blocks and centred sequences horizontally and vertically. For example, by mirroring Reflection, we obtain the following result: for $n\geq 3^{2t}+4t+2$, let $\mathbf{A}:[-t,t]\to \ZZ$ be a centred sequence and $B$, $C$, $X$ be blocks such that
\begin{itemize}
    \item $|B|=n$,
    \item $|C|=|X|$,
    \item $B$, $C$, $X$ are increasing,
    \item $C\prec B\prec \mathbf{A}\prec X$.
\end{itemize}
Then one can go from $C B \mathbf{A} X$ to $E \mathbf{D} \ol{C}$, where $\mathbf{D}:[-t,t]\to \ZZ$ is a decreasing centred sequence, $E$ is decreasing and $\mathbf{D}\prec E$ (see Figure~\ref{fig:reflect_rev}).

\begin{figure}
    \centering
    \includegraphics[scale=1.0]{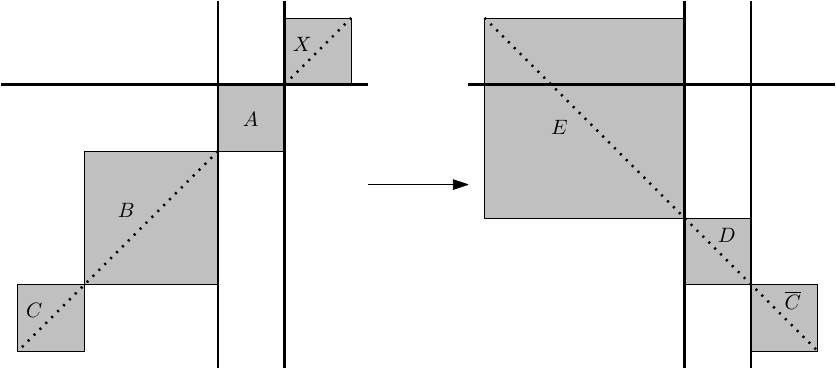}
    \caption{The mirrored version of Reflection.}
    \label{fig:reflect_rev}
\end{figure}

\section{The recursive step}

In this section, we show how to perform a certain recursive step which lies at the heart of our construction and is arguably its most difficult part. To state the result, we need some further definitions.

\begin{defn}
Given a block $B$ of size $n$ with sequence $B_1,\ldots,B_n$, the \textit{width} of $B$, denoted by $\width(B)$, is the largest integer $k$ such that there exist $i_1<i_2<\cdots<i_k$ with $B_{i_1}>B_{i_2}>\cdots>B_{i_k}$.
\end{defn}

Note that, by Dilworth's theorem, this is the same as the smallest integer $k$ such that $B_1,\ldots,B_n$ can be partitioned into $k$ disjoint increasing subsequences.

\begin{defn}
Given a block $B$, denote by $B^+$ the block corresponding to the subsequence of $B$ consisting of all the positive values. Denote by $B^-$ the block corresponding to the subsequence of $B$ consisting of all the negative values.
\end{defn}

\begin{defn}
Let $B$ be a block of size $n$ and $r\geq 0$ a real number. We say that $B$ is \textit{$r$-balanced} if 
\begin{enumerate}
    \item the block $B^-$ is increasing,
    \item for each $1\leq k\leq n$, the initial segment $B'=B|_{[k]}$ has $|B'^-|\geq r\cdot \width(B'^+)$.
\end{enumerate}
\end{defn}

The following key property of $r$-balanced blocks will be crucial to our main construction.

\begin{lem} \label{lem:decomp}
Suppose $B$ is a block that is $r$-balanced for some real $r\geq 0$. Then there is some $k$ such that a block of the form 
$$C_1 \cdots  C_k$$
can be obtained from $B$, where each block $C_i$ is increasing, $|C_i^-|\geq \floor{r}$ and $C_1^-\prec C_2^-\prec\cdots\prec C_k^-$.
\end{lem}

\begin{proof}
Note that if $B$ is $r$-balanced, then it is $r'$-balanced for any $r'\leq r$. Hence, without loss of generality, we can assume that $r$ is an integer. 

Let $A$ be a block of size $n$. We shall decompose $A$ into a collection of increasing subsequences. First set $I_0=[n]$. We define an increasing subsequence $i_{1,1},\ldots,i_{1,n_1}$ from $I_0$ as follows. Let $i_{1,1}$ be the smallest element of $I_0$. Given $i_{1,k}$, define $i_{1,k+1}$ to be the smallest element in $I_0$ larger than $i_{1,k}$ such that $A_{i_{1,k+1}}>A_{i_{1,k}}$. If no such $i_{1,k+1}$ exists, then we terminate the sequence and set $I_1=I_0\setminus\set{i_{1,1},\ldots,i_{1,n_1}}$. We then define $i_{2,1},\ldots,i_{2,n_2}$ similarly in terms of $I_1$ and set $I_2=I_1\setminus\set{i_{2,1},\ldots,i_{2,n_2}}$. Repeating this procedure until we reach $d$ with $I_d = \emptyset$, we see that we have partitioned $A$ into $d$ increasing subsequences $A_{i_{k,1}},\ldots,A_{i_{k,n_k}}$ for $k=1,\ldots,d$. 

We claim that $d=\width(A)$. In fact, we shall show that for any $1\leq m\leq n$,  if $k$ is the largest integer for which $i_{k,1}\leq m$, then $\width(A|_{[m]})=k$. Indeed, $\width(A|_{[m]})\leq k$, since we can cover $A|_{[m]}$ with $k$ increasing subsequences $A_{i_{j,1}},\ldots,A_{i_{j,n_j}}$ for $j=1,\ldots,k$. To show that $\width(A|_{[m]})\geq k$, we find a decreasing subsequence of length $k$. Let $j_k=i_{k,1}$. Since $i_{k,1}\not\in\set{i_{k-1,1},\ldots,i_{k-1,n_{k-1}}}$, there is some $l$ such that $i_{k-1,l}<j_k$ and $A_{i_{k-1,l}}>A_{j_k}$. Set $j_{k-1}=i_{k-1,l}$. By repeating this step, we eventually obtain a sequence $j_k>j_{k-1}>\cdots>j_1$ with $A_{j_k}<A_{j_{k-1}}<\cdots<A_{j_1}$, so we have the required decreasing subsequence of length $k$.

Now consider the above decomposition into increasing subsequences applied to $B^+$. If $k=\width(B^+)$, then we can partition $B^+$ into $k$ increasing subsequences $S_1,\ldots,S_k$, the $j$th of which starts from $B_{l_j}$ for some $l_j$. Set $C_1$ to be the concatenation of the first $r$ elements of $B^-$ and $S_1$. Set $C_2$ to be the concatenation of the next $r$ elements of $B^-$ and $S_2$ and so on, until $C_k$, which we take to be the concatenation of the remaining elements of $B^-$ and $S_k$. Note that, for each $j\leq k$, $\width(B|_{[l_j]})=j$, so, since $B$ is $r$-balanced, there are at least $jr$ elements of $B^-$ that come before $S_j$. Since $B^-$ is also increasing by the definition of $r$-balancedness, we can go from $B$ to $C_1\cdots C_k$, where each $C_i$ is increasing and $|C_i^-|\geq r$. Furthermore, $C_1^-\prec C_2^-\prec\cdots\prec C_k^-$, since they are the elements of $B^-$ in the right order.
\end{proof}

The following lemma is the main recursive step in our construction (see Figure~\ref{fig:recursive_pic}).

\begin{lem} \label{lem:main}
Let $d,n\geq 1$, set $T=3^{2t}$ and define sequences $\alpha_0,\alpha_1,\ldots$ and $\beta_0,\beta_1,\ldots$ inductively by
\begin{itemize}
    \item $\alpha_0=T+4t+2$,
    \item $\beta_0=0$,
    \item $\alpha_{i+1}=d\alpha_i+2Td^i+d$,
    \item $\beta_{i+1}=d\beta_i+\frac{d^{i+1}}{3T}$.
\end{itemize}
Suppose that $d\geq 9T$. Then, for each $k\geq 0$ and any centred sequence $\mathbf{I}:[-t,t]\to\ZZ$, there are increasing blocks $X,Y$ 
of some sizes with $X\prec \mathbf{I}\prec Y$ and $X\prec 0\prec Y$ such that one can go from $X \mathbf{I} Y$ to
$$L_k W_k \mathbf{A}_k B_k R_k,$$
where:
\begin{enumerate}
    \item $\mathbf{A}_k:[-t,t]\to\ZZ$,
    \item $W_k,L_k\succ 0\succ R_k$,
    \item if $k>0$, then $\mathbf{A}_k\prec B_k$, while if $k=0$, then $\mathbf{A}_k\succ B_k$,
    \item $W_k$ only contains elements from $Y$,
    \item $W_k$ is of the form 
    $$K_1 \cdots K_m M_m\cdots M_1,$$
    where $m=d^k$, each $K_i$ is decreasing, $|K_i|=n,|M_i|=1$ and 
    $$K_m\succ M_m\succ K_{m-1}\succ M_{m-1}\succ\cdots\succ K_1\succ M_1,$$
    \item $\width(B_k^+)\leq \alpha_k$,
    \item $|B_k^-|\geq \beta_k$,
    \item $B_k$ is $(\beta_k/\alpha_k)$-balanced.
\end{enumerate}
Furthermore, one can take $|X|,|Y|\leq 10d^{2k+1}n$.
\end{lem}

Though somewhat hidden in the detail of the induction hypothesis, the key point here is that one can make $B_k$ $r$-balanced for any given $r$ by choosing $k$ and $d$ appropriately. It is also worth noting that we will only make use of the lemma with $n = 1$, but our induction argument requires that we consider larger values of $n$.

\begin{figure}
    \centering
    \includegraphics[scale=0.9]{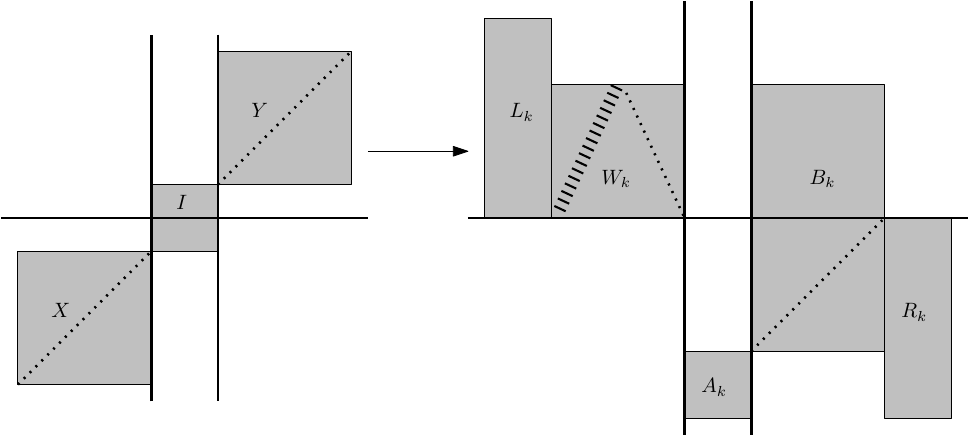}
    \caption{The main recursive step.}
    \label{fig:recursive_pic}
\end{figure}

\begin{proof}
We proceed by induction on $k$. We first prove the base case $k=0$. Set $|X|=n+1$ and $Y=C_1 C_2 C_3$, where $|C_1|=T+2t+1$, $|C_2|=2t+1$ and $|C_3|=n+1$. By reflecting $C_3$, we can go from $X \mathbf{I} Y$ to 
$$\ol{C_3}\ \mathbf{\ol{C_2}}\ \ol{C_1} D \ol{X},$$
where $D$ is the block we get by sorting $\mathbf{I}$ in decreasing order (see Figure~\ref{fig:k0}). 
But this completes the base case by setting $L_0=\emptyset$, $W_0=\ol{C_3}$, $\mathbf{A}_0=\mathbf{\ol{C_2}}$, $B_0=\ol{C_1} D^+$ and $R_0=D^- \ol{X}$ and noting that $\width(B_0)=|B_0|\leq T+4t+2=\alpha_0$.

\begin{figure}
    \centering
    \includegraphics[scale=1.0]{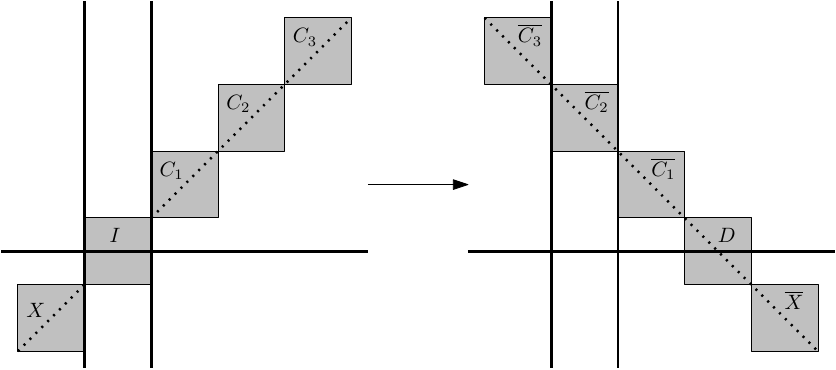}
    \caption{The case $k=0$.}
    \label{fig:k0}
\end{figure}

We now move on to the induction proper, assuming that the result holds for $k$ and then showing that it also holds for $k+1$. We first describe the overall strategy, which is driven by the need to maintain the structure of the block $W_{k+1}$, which has a head (the part containing the $K_i$) and a tail (the part containing the $M_i$). The first step is to perform the construction for $(d,n+1,k)$ $d$ times, thereby obtaining blocks $W_k^1,\ldots,W_k^d$. Next, we move the tails of the $W_k^i$ to the right. These will eventually be part of $B_{k+1}$ and their structure will allow us to boost the balancedness of that block. Finally, we peel off one layer from the heads of the $W_k^i$ and use it as ammunition to move elements into $B_{k+1}^-$. Once this is used up, it then forms the tail of $W_{k+1}$.\\

\noindent
\textbf{Step 1:} Repeating the induction hypothesis $d$ times.\\

\noindent
Write 
$$X \mathbf{I} Y=X' P_d P_{d-1} \cdots P_1 \mathbf{I} Q_1 Q_2 \cdots Q_d Y',$$
where $P_i$, $Q_i$ are blocks of suitable lengths for what follows. 
By our induction hypothesis (using $k$ and $n+1$ instead of $n$), we can go from $P_1 \mathbf{I} Q_1$ to $L_k^1 W_k^1 \mathbf{A}_k^1 B_k^1 R_k^1$. 
Since $P_i\prec P_1 \mathbf{I} Q_1\prec Q_i$ for $i>1$, we can therefore go from $X \mathbf{I} Y$ to
$$L_k^1 X' W_k^1 P_d\cdots P_2 \mathbf{A}_k^1 Q_2\cdots Q_d B_k^1 Y' R_k^1$$
(see Figure~\ref{fig:kinduct1}).

\begin{figure}
    \centering
    \includegraphics[scale=0.9]{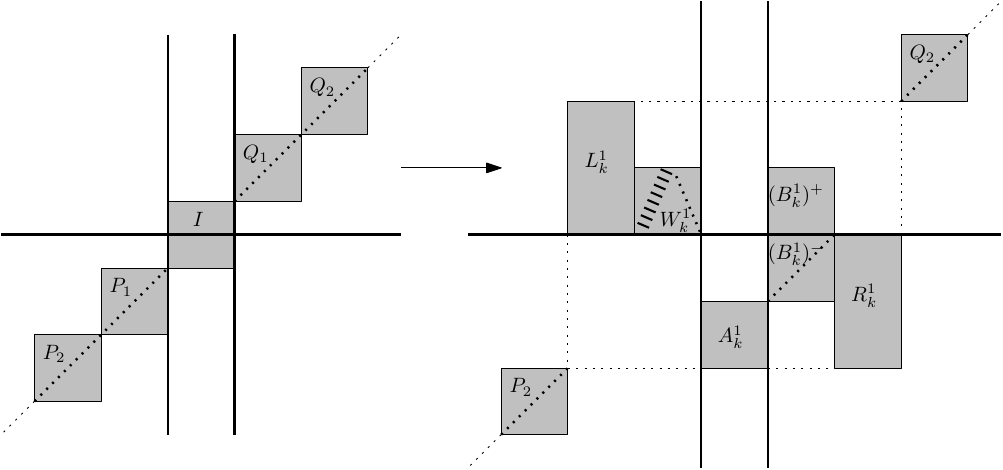}
    \caption{Step 1, part 1.}
    \label{fig:kinduct1}
\end{figure}

Repeating our induction hypothesis on $P_2 \mathbf{A}_k^1 Q_2$ to get $L_k^2 W_k^2 \mathbf{A}_k^2 B_k^2 R_k^2$ and so on, we can eventually go to
$$L' X' W_k^1\cdots W_k^d \mathbf{A}_k^d B' Y' R',$$
where $L'=L_k^1\cdots L_k^d$, $R'=R_k^d \cdots R_k^1$ and $B'=B_k^d \cdots B_k^1$ (see Figure~\ref{fig:kinduct2}). Since $W_k^i$ only contains elements from $Q_i$, we have $W_k^1 \prec W_k^2 \prec \cdots \prec W_k^d$, as depicted in Figure~\ref{fig:kinduct2}. 

\begin{figure}
    \centering
    \includegraphics[scale=1.0]{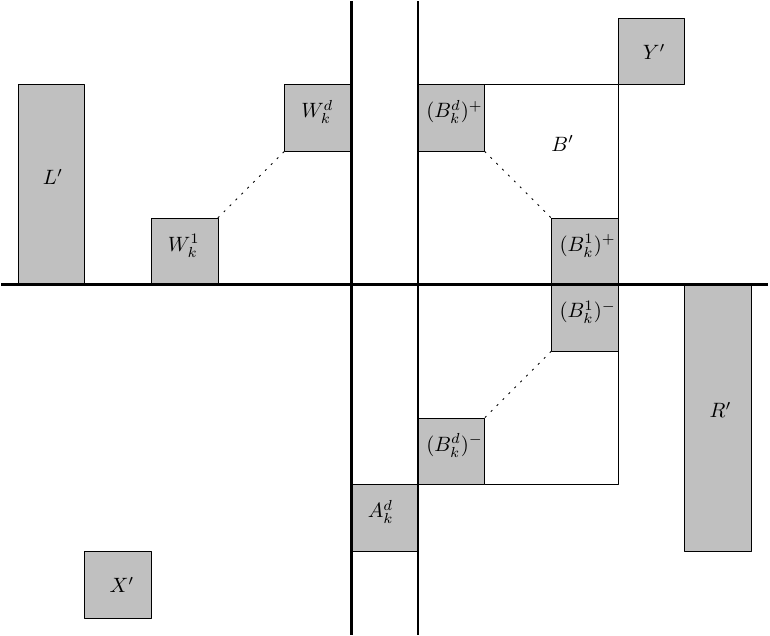}
    \caption{Step 1, part 2.}
    \label{fig:kinduct2}
\end{figure}

If $k=0$, then $\mathbf{A}_k^{i-1}\succ B_k^{i-1}$, $(B_k^i)^+$ only contains elements from $\mathbf{A}_k^{i-1} Q_i$ and $(B_k^i)^-$ only contains elements from $P_i$. Otherwise, if $k>0$, then $\mathbf{A}_k^{i-1}\prec B_k^{i-1}$, $(B_k^i)^+$ only contains elements from $Q_i$ and $(B_k^i)^-$ only contains elements from $P_i \mathbf{A}_k^{i-1}$. In either case, we have $(B_k^d)^+\succ (B_k^{d-1})^+\succ \cdots\succ (B_k^1)^+$ and $(B_k^d)^-\prec \cdots\prec (B_k^1)^-$, 
so $(B')^-$ is increasing. Furthermore, $|B'^-|=|(B_k^1)^-|+\cdots+|(B_k^d)^-|\geq d\beta_k$, $\width(B'^+)\leq \width((B_k^1)^+)+\cdots+\width((B_k^1)^+)\leq d\alpha_k$ and, since each $B_k^i$ is $(\beta_k/\alpha_k)$-balanced, $B'$ is $(\beta_k/\alpha_k)$-balanced.\\

\noindent
\textbf{Step 2:} Moving the tails of each $W_k$ to the right.\\

\noindent
Write $W_k^i=K_1^i K_2^i \cdots K_m^i M_m^i\cdots M_1^i$, where $m=d^k$ and $|K_j^i|=n+1$. Decompose $Y'$ into $S_1 \cdots S_m Y''$, where $|S_i|=T+d+4t+2$. 
Further decompose each $S_i$ into $T_i U_i$ where $|T_i|=T+2t+1$ and $|U_i|=d+2t+1$. Go to
$$L' X' W_k^1\cdots W_k^d  \mathbf{A}_k^d T_1 U_1 B' S_2\cdots S_m Y'' R'.$$
By performing suitable transpositions to move the points $M_j^i$, we can go to
$$L' X' W' M_m^1 M_m^2\cdots M_m^d M_{m-1}^1\cdots M_1^d \mathbf{A}_k^d T_1 U_1 B' S_2\cdots S_m Y'' R',$$
where $W'=K_1^1 K_2^1 \cdots K_m^1 K_1^2 \cdots K_m^d$ (see Figure~\ref{fig:kinduct3}). 

\begin{figure}
    \centering
    \includegraphics[scale=1.0]{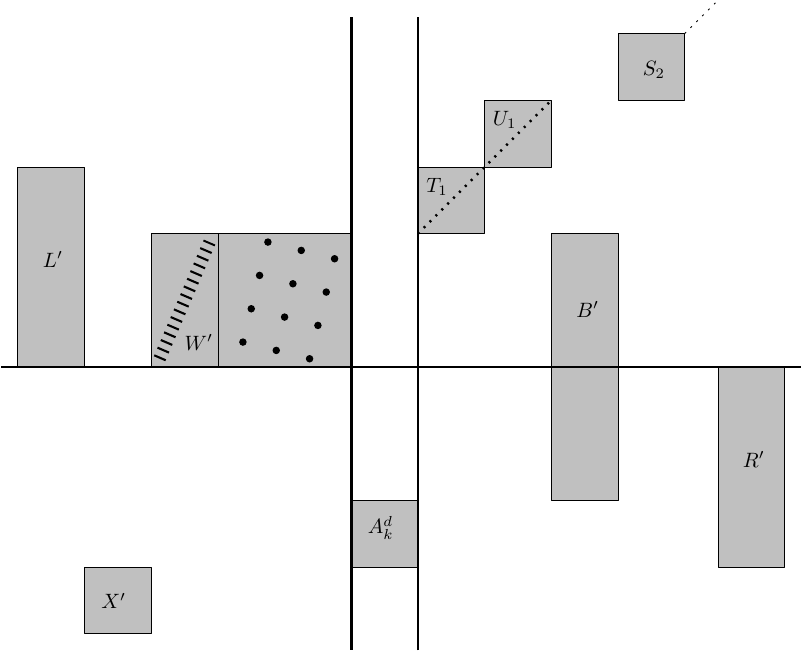}
    \caption{Step 2, part 1.}
    \label{fig:kinduct3}
\end{figure}

We would now like to move all the $M_j^i$ to the right. Write $T_1=T_1' J_1$, where $|J_1|=2t+1$ and $|T_1'|=T$. By shifting $J_1$, we can go to
$$L' X' W' M_m^1 \cdots M_1^d  \mathbf{J}_1 C_1 U_1 B' S_2\cdots S_m Y'' R',$$
where $\mathbf{J}_1:[-t,t]\to \ZZ$ is increasing and $C_1$ is decreasing, noting that $|C_1|=|T_1'|+|A_k^d|=T+2t+1$. If $k=0$, then $\mathbf{A}_0^d\succ 0$ implies that $C_1^-=\emptyset$. Otherwise, $\mathbf{A}_k^d\prec B'$, so $C_1^-\prec B'$ and we can move $C_1^-$ all the way to the right to be part of $R'$. Thus, we can assume that $C_1\succ 0$ and $|C_1|\leq T+2t+1$ (see Figure~\ref{fig:kinduct4}). 

\begin{figure}
    \centering
    \includegraphics[scale=1.0]{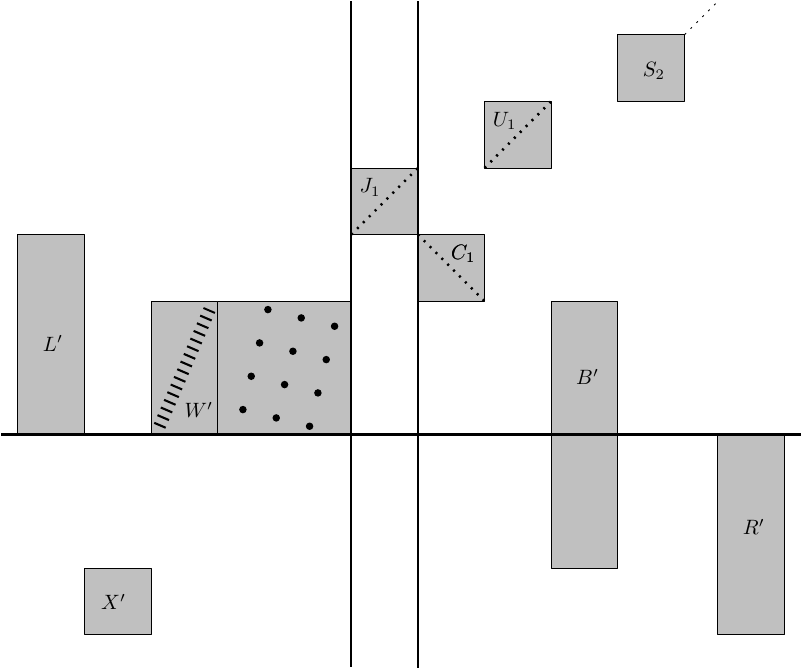}
    \caption{Step 2, part 2.}
    \label{fig:kinduct4}
\end{figure}

We can go from $C_1 U_1$ to $U_1 C_1$. Decompose $U_1=U_1' U_1''$, where $|U_1'|=2t+1$ and $|U_1''|=d$. Then, perform the flip $[-d-t,d+3t+1]$ to go to 
$$L' X' W' M_m^1 \cdots M_2^d  \ol{U_1''}\ \mathbf{\ol{U_1'}}\ \ol{J_1} N_1 C_1 B' S_2\cdots S_m Y'' R',$$
where $N_1=M_1^d M_1^{d-1}\cdots M_1^1$. We can then move $\ol{U_1''}$ to the left to be absorbed by $L'$ (see Figure~\ref{fig:kinduct5}).

\begin{figure}
    \centering
    \includegraphics[scale=1.0]{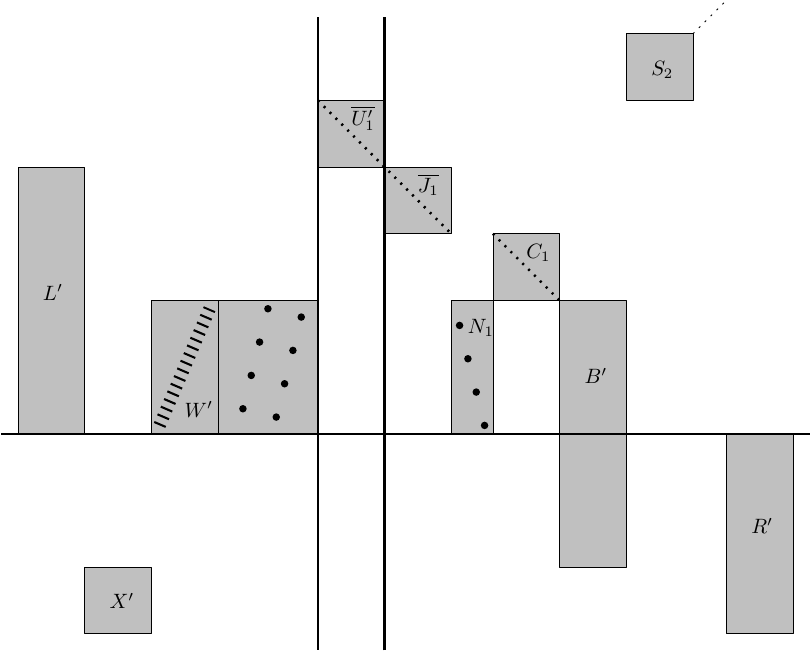}
    \caption{Step 2, part 3.}
    \label{fig:kinduct5}
\end{figure}

Repeating this process, we can go to
$$L' X' W' \mathbf{\ol{U_m'}} C B' Y'' R',$$
where $C=\ol{J_m} N_m C_m \cdots \ol{J_1} N_1 C_1\succ 0$ and $N_i=M_i^d M_i^{d-1}\cdots M_i^1$ (see Figure~\ref{fig:kinduct6}).

\begin{figure}
    \centering
    \includegraphics[scale=1.0]{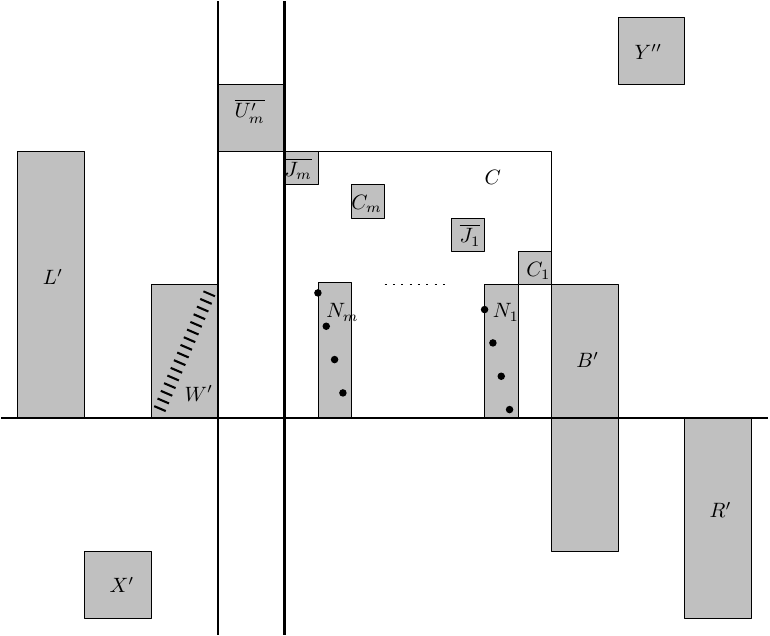}
    \caption{Step 2, part 4.}
    \label{fig:kinduct6}
\end{figure}

Before continuing with the construction, let us analyse the width of the block $C$. Observe that the width of interlaced blocks is at most the sum of the widths of the individual blocks, so 
$$\width(C)\leq \sum_{i=1}^m\width(\ol{J_i} C_i)+\width(N_m\cdots N_1).$$
We have $\width(\ol{J_i} C_i)\leq |J_i|+|C_i|\leq T+4t+2$. Moreover, $N_m\cdots N_1$ can be covered by $m+d-1$ increasing subsequences, so that $\width(N_m\cdots N_1)\leq m+d-1$. Overall, we have
$$\width(C)\leq m(T+4t+2)+m+d-1\leq 2mT+d.$$

\noindent
\textbf{Step 3:} Forming the tail of $W_{k+1}$.\\

\noindent
Move $X'$ and $Y''$ towards the centre. In this step, we will largely focus on the centred subsequence $W' X' \mathbf{\ol{U_m'}} Y''$, only coming back to the full sequence towards the end.

Expand $W'=K_1\cdots K_{md}$, where each $K_i$ is decreasing and $|K_i|=n+1$. Write $K_i=K_i' O_i$, where $|K_i'|=n$ and $|O_i|=1$. Then we can go from $W'$ to
$$K_1'\cdots K_{md}' O_1 \cdots O_{md}.$$
Write $X'=X'' P_p P_{p-1}\cdots P_1$ and $Y''=Q_1\cdots Q_p$, where $|P_i|=|Q_i|=1$, $|X''|=2t+1$ and $p=\floor{\frac{md-T}{T+4t+2}}\geq 1$, by using $d\geq 9T$ (note that we have now fixed the original sizes of $X$ and $Y$). 
Set $W''=K_1'\cdots K_{md}'$. We shall group the $O_i$ as
$$O_1\cdots O_{md}=F G_p H_p G_{p-1} H_{p-1}\cdots G_1 H_1,$$
where $|G_i|=2t+1,|H_i|=T+2t+1$ and $F$ is what remains, so that $|F|\geq T$ from the definition of $p$.

By rearranging the blocks, we can go to 
$$W'' X'' F P_p G_p H_p\cdots P_1 G_1 H_1 \mathbf{\ol{U_m'}} Q_1 \cdots Q_p.$$

By reflecting $P_1$ in the mirrored sense, we can go from $P_1 G_1 H_1 \mathbf{\ol{U_m'}} Q_1$ to $Q_1 \ol{U_m'}\ \ol{H_1}\ \mathbf{\ol{G_1}} P_1$. Then, moving $Q_1 \ol{U_m'}$ to the left, we can go to 
$$Q_1 \ol{U_m'}  W'' X'' F P_p G_p H_p\cdots P_2 G_2 H_2 \ol{H_1}\ \mathbf{\ol{G_1}} P_1 Q_2\cdots Q_p$$
(see Figure~\ref{fig:kinduct7}). 

\begin{figure}
    \centering
    \includegraphics[scale=0.8]{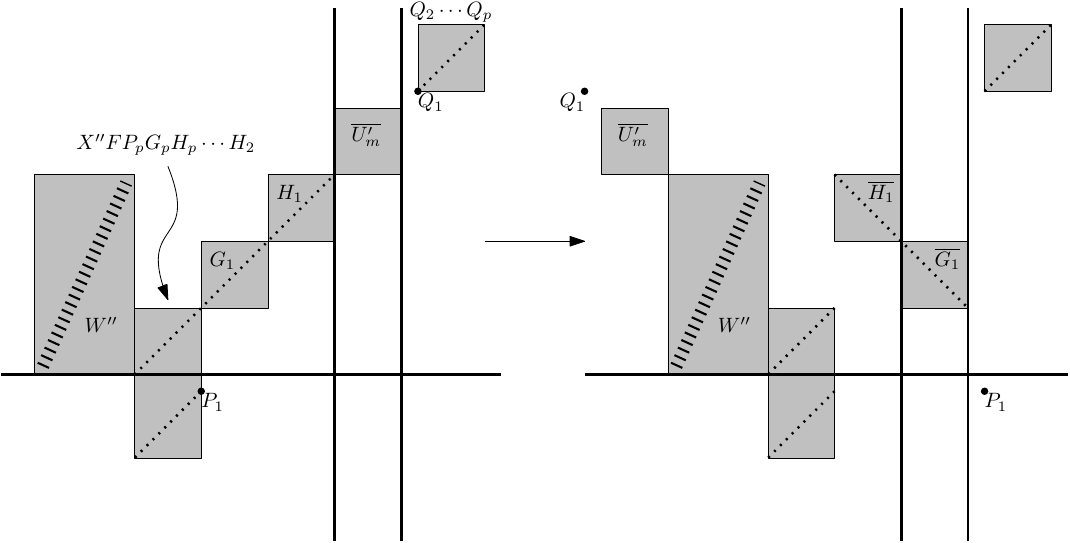}
    \caption{Step 3, part 1.}
    \label{fig:kinduct7}
\end{figure}

Move $\ol{H_1}$ to be between $W''$ and $X''$ and move $P_1$ to the right. Reflecting $P_2$ similarly, we can go from $P_2 G_2 H_2 \mathbf{\ol{G_1}} Q_2$ to $Q_2 \ol{G_1}\ \ol{H_2}\ \mathbf{\ol{G_2}} P_2$
(see Figure~\ref{fig:kinduct8}). 
Rearranging the blocks, we obtain
\begin{align*}
    & Q_1 Q_2 \ol{U_m'}  W'' \ol{H_1}\ \ol{G_1}\ \ol{H_2} X'' F P_p G_p H_p\cdots P_3 G_3 H_3 \mathbf{\ol{G_2}} Q_3\cdots Q_p P_2 P_1.
\end{align*}

\begin{figure}
    \centering
    \includegraphics[scale=0.8]{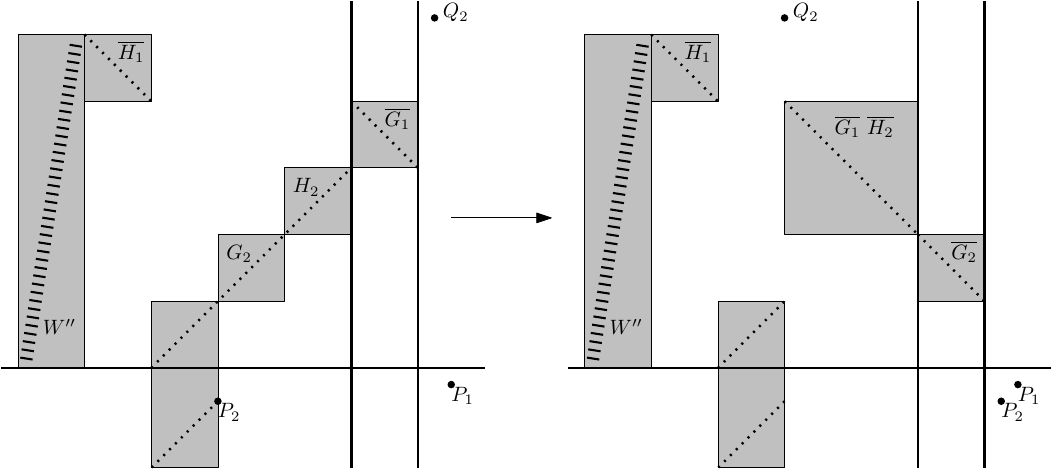}
    \caption{Step 3, part 2.}
    \label{fig:kinduct8}
\end{figure}

Repeating all the way to $P_p G_p H_p$, we obtain
$$Q \ol{U_m'}  W'' \ol{H_1}\ \ol{G_1}\cdots \ol{H_p} X'' F \mathbf{\ol{G_p}} P,$$
where $P=P_p\cdots P_1$ and $Q=Q_1\cdots Q_p$ are increasing. By shifting $X''$ in the mirrored sense, we obtain
$$Q \ol{U_m'}  W'' \ol{H_1}\ \ol{G_1}\cdots \ol{H_p}\ \ol{G_p}\ \ol{F} \mathbf{X''} P$$
(see Figure~\ref{fig:kinduct9}).

\begin{figure}
    \centering
    \includegraphics[scale=0.9]{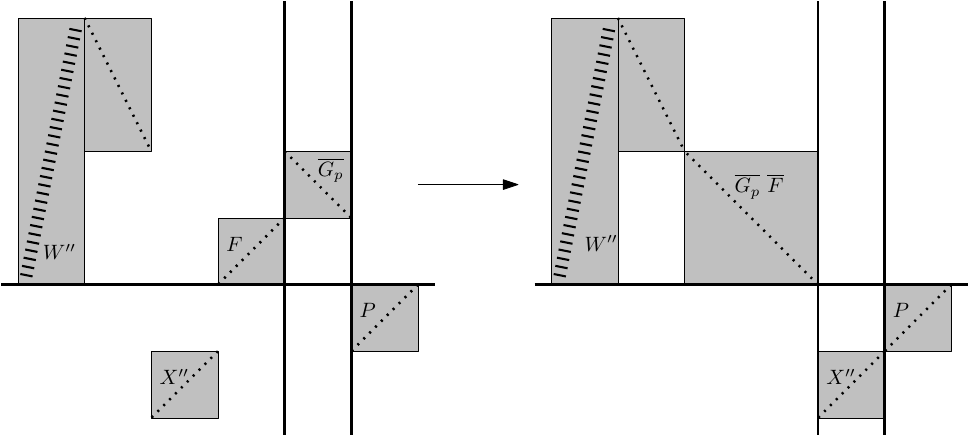}
    \caption{Step 3, part 3.}
    \label{fig:kinduct9}
\end{figure}

Overall, we managed to go from the original sequence $X \mathbf{I} Y$ to
$$L' Q \ol{U_m'}  W'' O_{md}\cdots O_1 \mathbf{X''} P C B' R'.$$
Now set $L_{k+1}=L' Q \ol{U_m'}$, $W_{k+1}=W'' O_{md}\cdots O_1$, $\mathbf{A}_{k+1}=\mathbf{X''}$, $B_{k+1}=P C B'$ and $R_{k+1}=R'$. This concludes the construction, but we still need to check that the required conditions hold. Conditions (1), (2), (3), (4) and (5) are easy to check, so it only remains to show that $B_{k+1}$ satisfies conditions (6), (7) and (8). 

For $l \geq 0$, let $\beta_l' = \frac{d^{l +1}}{3T}$ and $\alpha_l'=2d^lT+d$. In particular, $\beta_k'=\frac{md}{3T}\leq \floor{\frac{md-T}{T+4t+2}}=p$ and $\alpha_k'=2mT+d$, where the inequality holds since $d \ge 9T$. We have $\width(B_{k+1}^+)\leq \width(C)+\width(B'^+)\leq \alpha_k'+d\alpha_k=\alpha_{k+1}$, proving (6). We also have that $|B_{k+1}^-|=|P|+|B'^-|\geq \beta_k'+d\beta_k=\beta_{k+1}$, proving (7). To prove (8), we need the following claim.

\vspace{3mm}
\noindent
{\bf Claim.} {\it For all $l \geq 0$, $\beta_l'/\alpha_l'>\beta_{l+1}/\alpha_{l+1}>\beta_l/\alpha_l$.}

\begin{proof}
We will show by induction on $l$ that $\beta_l'/\alpha_l'>\beta_l/\alpha_l$ and the full result will fall out as a consequence. The base case $l=0$ is trivial, since $\beta_0=0$. To show the result for $l+1$, note that, since $\frac{\beta_{l+1}}{\alpha_{l+1}}=\frac{\beta_l'+d\beta_l}{\alpha_l'+d\alpha_l}$ and $\beta_l'/\alpha_l'>\beta_l/\alpha_l$, 
we have $\beta_l'/\alpha_l'>\beta_{l+1}/\alpha_{l+1}>\beta_l/\alpha_l$. But $\frac{\beta_{l+1}'}{\alpha_{l+1}'}=\frac{d^{l+2}}{3T(2d^{l+1}T+d)}\geq \frac{d^{l+1}}{3T(2d^{l}T+d)}=\frac{\beta_{l}'}{\alpha_{l}'}$, so that $\beta_{l+1}'/\alpha_{l+1}'>\beta_{l+1}/\alpha_{l+1}$, as required. 
\end{proof}

Suppose now that $J$ is an initial segment of $B_{k+1}$. If $J$ is a subblock of $P$, then $J^+=\emptyset$. If $J$ contains $P$ and is a subblock of $P C$, then $|J^-|=p\geq\beta_k'$ and $\width(J^+)\leq \width(C)\leq \alpha_k'$, so, by the claim, $|J^-|\geq (\beta_{k+1}/\alpha_{k+1})\width(J^+)$. Suppose then that $J=P C J'$, where $J'$ is a non-empty initial segment of $B'$. Since $B'$ is $(\beta_k/\alpha_k)$-balanced, we have $|J'^-|\geq (\beta_k/\alpha_k)\width(J'^+)$. But $|J^-|\geq \beta_k'+|J'^-|$ and $\width(J^+)\leq \alpha_k'+\width(J'^+)$, so, since $\width(J'^+)\leq \width(B'^+)\leq d\alpha_k$, we have
$$\frac{|J^-|}{\width(J^+)}\geq \frac{\beta_k'+(\beta_k/\alpha_k)\width(J'^+)}{\alpha_k'+\width(J'^+)}\geq \frac{\beta_k'+d\beta_k}{\alpha_k'+d\alpha_k}=\frac{\beta_{k+1}}{\alpha_{k+1}}.$$
Note that here, in the second inequality, we used that, since $\beta_k'/\alpha_k' > \beta_k/\alpha_k$, the function $(\beta_k'+(\beta_k/\alpha_k)x)/(\alpha_k'+x)$ is decreasing in $x$.
Since $B_{k+1}^-$ is increasing, this implies that $B_{k+1}$ is $(\beta_{k+1}/\alpha_{k+1})$-balanced, proving (8).

Finally, we give bounds on $|X|$ and $|Y|$. Suppose, for given $d,n,k$, that the construction gives $|X|=x_{d,n,k}$ and $|Y|=y_{d,n,k}$. Then $x_{d,n,0}=n+1$ and $y_{d,n,0}=T+4t+3+n$. Moreover, following the construction, we have that, for $k\geq 0$, 
\begin{align*}
    x_{d,n,k+1} &= dx_{d,n+1,k}+p+2t+1\leq dx_{d,n+1,k}+\frac{d^{k+1}}{2T}+2t+1\\
    y_{d,n,k+1} &= dy_{d,n+1,k}+m(T+d+4t+2)+p\leq dy_{d,n+1,k}+2d^{k+1},
\end{align*}
so we get the (loose) bounds $x_{d,n,k},y_{d,n,k}\leq 10d^{2k+1}n$.
\end{proof}

\section{The full construction}

We now proceed to the full construction. Set $T=3^{2t}$. Start with the identity centred sequence $\mathbf{I}:[a,b]\to\ZZ$ for suitable $a<0<b$ to be decided later. 
Perform the flip $[-t,3t+1]$ to obtain 
$$X' X \mathbf{I}' J Y,$$
where $\mathbf{I}':[-t,t]\to\ZZ$, $|J|=2t+1$ and $X',X$ are of suitable lengths to be decided later. $J$ is an important piece that we will have to bring back to the centre at the end. For now, we set it aside by moving it to the right to get
$$X' X \mathbf{I}' Y J$$
(see Figure~\ref{fig:construction1}).

\begin{figure}
    \centering
    \includegraphics[scale=1.0]{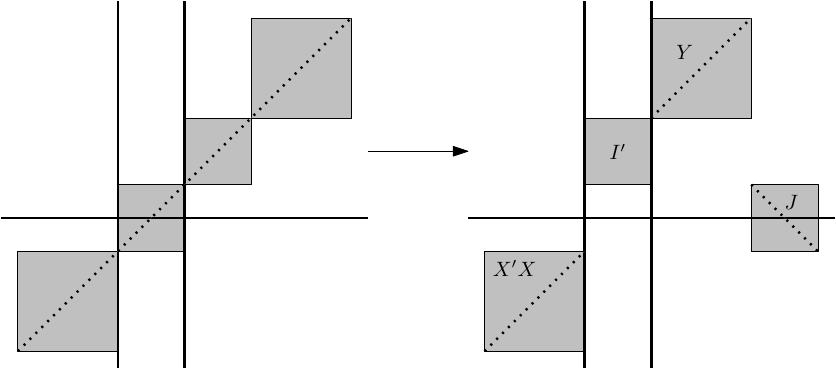}
    \caption{Starting the construction.}
    \label{fig:construction1}
\end{figure}

Applying Lemma~\ref{lem:main} to $X \mathbf{I}' Y$ with $n=1$ and some $d$ (with suitably chosen $a,b$), we can go to
$$X' L_k W_k \mathbf{A}_k B_k R_k J,$$
where $B_k$ is $r=(\beta_k/\alpha_k)$-balanced. 

We now show that for a suitable choice of $d$ and $k$, we can make $r$ arbitrarily large. Indeed, we have $\alpha_1\leq 3dT$,
$$\alpha_{i+1}\leq d\alpha_i+3Td^i$$
for $i\geq 1$ and
$$\beta_{i+1}\geq d\beta_i+\frac{d^{i+1}}{3T},$$
so that $\alpha_k\leq 3T d^k+3kTd^{k-1}$ and $\beta_k\geq \frac{k}{3T}d^k$. Hence, $r\geq \frac{k}{3T(3T+3kT/d)}$. Setting $k=d$, we see that $r\geq \frac{d}{18T^2}$ can be taken arbitrarily large. In practice, we will set $d=100T^3$, so that $r\geq 3T+1$. 

By Lemma~\ref{lem:decomp}, we can go from $B_k$ to $C_1\cdots C_{m-1} C_m'$ where each block $C_i$ and $C_m'$ is increasing and $|C_i^-|,|C_m'^-|\geq \floor{r}\geq 3T$. We can decompose $C_m'$ into $C_m'' Z  C_m'''$, where $C_m''\prec Z\prec 0\prec C_m'''$, $|Z|\geq T$ and $|C_m''|\geq 2T$. We can go from $C_m'' Z C_m'''$ to $C_m'' C_m''' Z$. Now set $C_m=C_m' C_m''$. Thus, we have gone from $B_k$ to $C_1\cdots C_m Z$, where $C_1^-\prec C_2^-\prec \cdots\prec C_m^-\prec Z\prec 0$ and $|C_i^-|\geq 2T$ for each $i$ (see Figure~\ref{fig:construction_C_decomp}).

\begin{figure}
    \centering
    \includegraphics[scale=1.0]{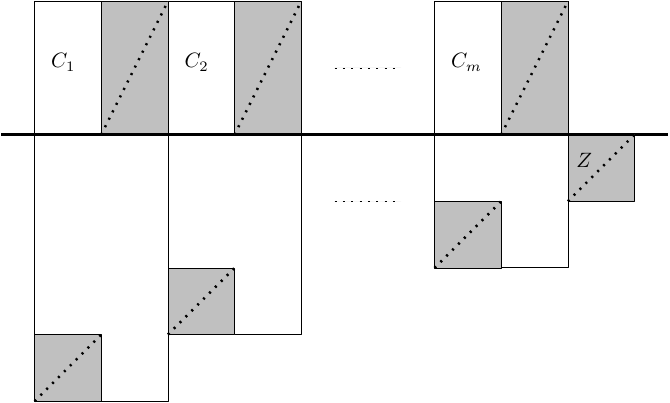}
    \caption{Transforming $B_k$.}
    \label{fig:construction_C_decomp}
\end{figure}

For each $i$, decompose $C_i$ into $D_i E_i F_i$, where $D_i\prec E_i\prec 0\prec F_i$ and $|E_i|=2t+1$, so that $|D_i|\geq T+2t+1$. Decompose $X'$ as $X'=X_m\cdots X_1$ with $|X_i|=|F_i|$, thus fixing the value of $a$. 

Note that $\mathbf{A}_k\prec C_1$. By bringing $X_1$ towards the centre and reflecting $F_1$, we can go from $X_1 \mathbf{A}_k D_1 E_1 F_1$ to 
$$\ol{F_1}\ \mathbf{\ol{E_1}} P_1,$$
for some decreasing $P_1\prec E_1$ (see Figure~\ref{fig:construction_moveC}).

\begin{figure}
    \centering
    \includegraphics[scale=1.0]{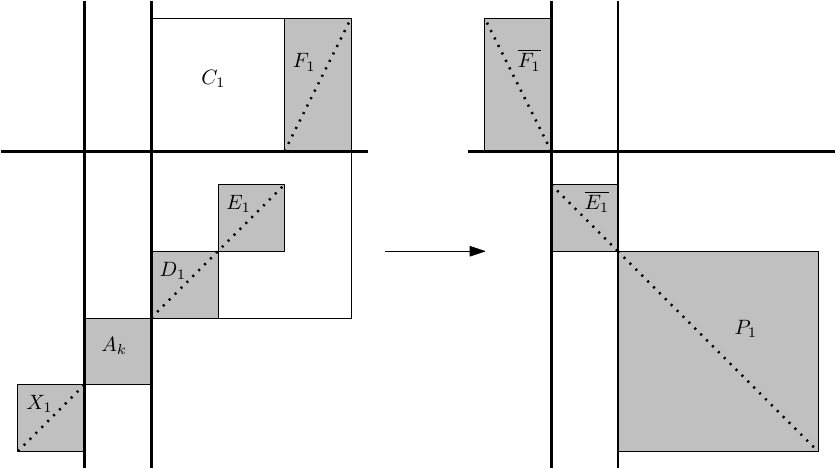}
    \caption{Moving $C_1^+$ across.}
    \label{fig:construction_moveC}
\end{figure}

Note that $P_1$ consists of elements from $X_1$, $A_k$ and $D_1$, so $P_1\prec Z J$. Set $P_1$ aside by moving it all the way to the right, just before $Z$. 
Repeating this process for all the $C_i$, we can go from 
$$X' L_k W_k \mathbf{A}_k B_k R_k J$$
to
$$L_k W_k \ol{F_1}\ \ol{F_2}\cdots \ol{F_m}\ \mathbf{\ol{E_m}} P_m \cdots P_1 Z R_k J.$$
Bringing $Z J$ to the left, we can go to $$L_k W_k \ol{F_1} \cdots \ol{F_m}\ \mathbf{\ol{E_m}} Z J P_m\cdots P_1 R_k$$
(see Figure~\ref{fig:construction_final}).

\begin{figure}
    \centering
    \includegraphics[scale=1.0]{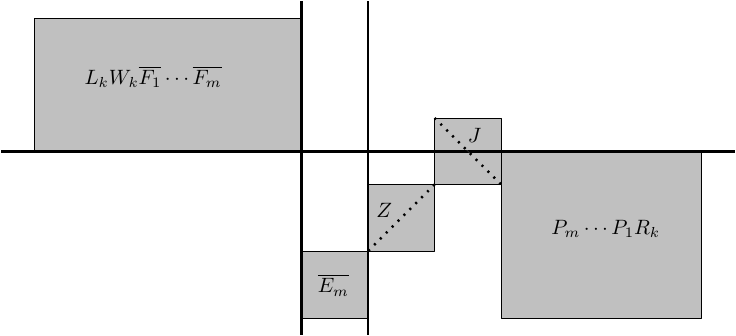}
    \caption{Almost done!}
    \label{fig:construction_final}
\end{figure}

Finally, shift $J$ to go from $\mathbf{\ol{E_m}} Z J$ to $\mathbf{J} Z'$, where $Z'\prec 0$. 
Now sort $L_k W_k \ol{F_1} \cdots \ol{F_m}$ and $Z' P_m\cdots P_1 R_k$ in decreasing order to get the reverse of the identity. Since this must be a centred sequence $[a,b]\to\ZZ$, we have $a=-b$. 
Moreover, we have $b=3t+1+|Y|$ and, by Lemma~\ref{lem:main}, $|Y|\leq 10d^{2k+1}n$ with $k=d=100T^3$ and $n=1$, giving $b= 2^{2^{O(t)}}$, as required.

\section{Concluding remarks}

The main problem left open by our work is whether there are configurations like ours which can be realised by lines. In the literature, such configurations are called {\it stretchable}. While we suspect that our configurations, as given, are not stretchable, we firmly believe, and conjecture below, that there should be stretchable configurations, perhaps even suitable variants of our construction, where the number of points on either side of each line differs by at least $k$ for any given natural number $k$.

\begin{conj}
There exists an unbounded function $f: \mathbb{N} \rightarrow \mathbb{N}$ such that for every natural number $n$ there is a set of $n$ points for which the number of points on either side of each line determined by the set differ by at least $f(n)$. 
\end{conj}

\noindent
It may also be that we can take $f(n) = c \log \log n$ for some $c > 0$, matching our Theorem~\ref{thm:main1}. However, at present, we do not even know how to find point sets where the number of points on either side of each line determined by the set differ by at least $3$.

Our Theorem~\ref{thm:main1} also implies that another result of Pinchasi is essentially best possible, in that any improvement must somehow take into account the fact that one is dealing with lines rather than pseudolines. If we define $f(k)$ to be the maximum size of a finite point set in the plane, not  contained in a line, with the property that there is no line through at least two of these points with at least $k$ points on either side of this line, then Pinchasi~\cite{P03} showed that $f(k) < 2k + C \log \log k$ for some absolute constant $C$. Once again, this follows from a statement about allowable sequences and therefore holds in the broader context of generalised configurations. Moreover, in this form, the result easily implies Theorem~\ref{thm:Pin}, since if $n = 2k + C \log \log k$, then either every point is on some pseudoline, in which case we are done, or there exists a pseudoline with at least $k$ points on either side and the number of points on each side differ by at most $C\log \log k$. By the same argument, any improved bound for this result would give an analogous improvement to Theorem~\ref{thm:Pin}, but our Theorem~\ref{thm:main1} shows that this is already best possible.

One may also study analogues of Kupitz's question in higher dimensions. In three dimensions, the most natural analogue is whether there is a fixed natural number $k$ such that every finite point set has a plane through at least three of its points where the number of points on either side of this plane differ by at most $k$. Curiously, we do not even know the answer under the far less prescriptive condition that the plane only pass through at least two points. By a simple projection argument, we know that an analogue of Pinchasi's result holds in this context, namely, that if our point set has $n$ points, then there is a plane through at least two points where the number of points on either side of this plane differ by at most $C \log \log n$. However, we do not expect this to be tight.

\end{document}